\documentclass[10pt,twoside]{article}
\usepackage{hyperref}
\usepackage{mathptmx}
\usepackage{amsmath}
\usepackage{amssymb}
\usepackage{amsthm}
\usepackage{lmodern}
\hypersetup{colorlinks,citecolor=blue,linkcolor=blue}

\def\YEAR{\year}\newcount\VOL\VOL=\YEAR\advance\VOL by-1995
\def\firstpage{1}\def\lastpage{1000}
\def\received{}\def\revised{}
\def\communicated{}

\makeatletter
\def\magnification{\afterassignment\m@g\count@}
\def\m@g{\mag=\count@\hsize6.5truein\vsize8.9truein\dimen\footins8truein}
\makeatother

\font\eightrm=cmr8
\font\caps=cmcsc10                    
\font\Caps=cmcsc10 scaled \magstep1   

\makeatletter
\@specialpagefalse\headheight=8.5pt
\def\DocMath{}
\renewcommand{\@evenhead}{%
    \ifnum\thepage>\lastpage\rlap{\thepage}\hfill%
    \else\rlap{\thepage}\slshape\leftmark\hfill{\caps\SAuthor}\hfill\fi}%
\renewcommand{\@oddhead}{%
    \ifnum\thepage=\firstpage{\DocMath\hfill\llap{\thepage}}%
    \else{\slshape\rightmark}\hfill{\caps\STitle}\hfill\llap{\thepage}\fi}%
\makeatother

\def\TSkip{\bigskip}
\newbox\TheTitle{\obeylines\gdef\GetTitle #1
\ShortTitle  #2
\SubTitle    #3
\Author      #4
\ShortAuthor #5
\EndTitle
{\setbox\TheTitle=\vbox{\baselineskip=20pt\let\par=\cr\obeylines%
\halign{\centerline{\Caps##}\cr\noalign{\medskip}\cr#1\cr}}%
	\copy\TheTitle\TSkip\TSkip%
\def\next{#2}\ifx\next\empty\gdef\STitle{#1}\else\gdef\STitle{#2}\fi%
\def\next{#3}\ifx\next\empty%
    \else\setbox\TheTitle=\vbox{\baselineskip=20pt\let\par=\cr\obeylines%
    \halign{\centerline{\caps##} #3\cr}}\copy\TheTitle\TSkip\TSkip\fi%
\centerline{\caps #4}\TSkip\TSkip%
\def\next{#5}\ifx\next\empty\gdef\SAuthor{#4}\else\gdef\SAuthor{#5}\fi%
\ifx\received\empty\relax
    \else\centerline{\eightrm Received: \received}\fi%
\ifx\revised\empty\TSkip%
    \else\centerline{\eightrm Revised: \revised}\TSkip\fi%
\ifx\communicated\empty\relax
    \else\centerline{\eightrm Communicated by \communicated}\fi\TSkip\TSkip%
\catcode'015=5}}\def\Title{\obeylines\GetTitle}
\def\Abstract{\begingroup\narrower
    \parskip=\medskipamount\parindent=0pt{\caps Abstract. }}
\def\EndAbstract{\par\endgroup\TSkip}

\long\def\MSC#1\EndMSC{\def\arg{#1}\ifx\arg\empty\relax\else
     {\par\narrower\noindent%
     2000 Mathematics Subject Classification: #1\par}\fi}

\long\def\KEY#1\EndKEY{\def\arg{#1}\ifx\arg\empty\relax\else
	{\par\narrower\noindent Keywords and Phrases: #1\par}\fi\TSkip}

\newbox\TheAdd\def\Addresses{\vfill\copy\TheAdd\vfill
    \ifodd\number\lastpage\vfill\eject\phantom{.}\vfill\eject\fi}
{\obeylines\gdef\GetAddress #1
\Address #2 
\Address #3
\Address #4
\EndAddress
{\def\xs{4.3truecm}\parindent=0pt
\setbox0=\vtop{{\obeylines\hsize=\xs#1\par}}\def\next{#2}
\ifx\next\empty 
     \setbox\TheAdd=\hbox to\hsize{\hfill\copy0\hfill}
\else\setbox1=\vtop{{\obeylines\hsize=\xs#2\par}}\def\next{#3}
\ifx\next\empty 
     \setbox\TheAdd=\hbox to\hsize{\hfill\copy0\hfill\copy1\hfill}
\else\setbox2=\vtop{{\obeylines\hsize=\xs#3\par}}\def\next{#4}
\ifx\next\empty\ 
     \setbox\TheAdd=\vtop{\hbox to\hsize{\hfill\copy0\hfill\copy1\hfill}
                \vskip20pt\hbox to\hsize{\hfill\copy2\hfill}}
\else\setbox3=\vtop{{\obeylines\hsize=\xs#4\par}}
     \setbox\TheAdd=\vtop{\hbox to\hsize{\hfill\copy0\hfill\copy1\hfill}
	        \vskip20pt\hbox to\hsize{\hfill\copy2\hfill\copy3\hfill}}
\fi\fi\fi\catcode'015=5}}\gdef\Address{\obeylines\GetAddress}

\DeclareMathOperator{\GL}{GL}

\begin{document}
\Title
Hermitian lattices and bounds in $K$-theory 
of algebraic integers
\ShortTitle
Hermitian lattices and bounds in $K$-theory 
\SubTitle   
\Author 
Eva Bayer-Fluckiger\footnote{First and third authors partially supported by SNSF, Project number 200020-121647}, Vincent Emery\footnote{Second author supported by SNSF, Project number PZ00P2-148100} and Julien Houriet\footnotemark[1]
\ShortAuthor 
E. Bayer-Fluckiger, V. Emery and J. Houriet
\EndTitle
\Abstract 
  Elaborating on a method of Soul\'e, and using better estimates for the
  geometry of Hermitian lattices, we improve the upper bounds for the torsion
  part of the $K$-theory of the rings of integers of  number fields.
\EndAbstract
\MSC 
11R70 (primary); 11E39 (secondary)
\EndMSC
\KEY 
$K$-theory of number fields, Hermitian lattices,
\EndKEY
\Address 
SB--MATHGEOM--CSAG
EPF Lausanne
B\^atiment MA, Station 8
CH-1015 Lausanne
Switzerland
\Address
eva.bayer@epfl.ch
vincent.emery@math.ch
julien.houriet@gmail.com
\Address
\Address
\EndAddress

\newtheorem{theorem}{Theorem}[section]
\newtheorem{lemma}[theorem]{Lemma}
\newtheorem{prop}[theorem]{Proposition}
\newtheorem{cor}[theorem]{Corollary}
\newtheorem{conjecture}[theorem]{Conjecture}
\theoremstyle{definition}
\newtheorem{definition}[theorem]{Definition}
\theoremstyle{remark}
\newtheorem{rmk}{Remark}
\newtheorem{example}[theorem]{Example}

\numberwithin{equation}{section}

\newcommand{\tors}{\mathrm{tors}}
\newcommand{\free}{\mathrm{free}}
\newcommand{\im}{\mathrm{im}}

\newcommand{\N}{\mathbb N}
\newcommand{\Z}{\mathbb Z}
\newcommand{\Q}{\mathbb Q}
\newcommand{\R}{\mathbb R}
\newcommand{\C}{\mathbb C}
\newcommand{\A}{\mathbb A}
\newcommand{\f}{\mathfrak f}

\renewcommand{\O}{\mathcal O}
\newcommand{\D}{D}
\newcommand{\q}{\mathfrak q}
\newcommand{\p}{\mathfrak p}


\newcommand{\Af}{\A_\mathrm{f}}

\newcommand{\V}{V}
\newcommand{\Vf}{\V_{\mathrm{f}}}
\newcommand{\Sf}{S_{\mathrm{f}}}
\newcommand{\Vi}{\V_\infty}

\newcommand{\tA}{\mathrm{A}}
\newcommand{\tB}{\mathrm{B}}
\newcommand{\tC}{\mathrm{C}}
\newcommand{\tD}{\mathrm{D}}
\newcommand{\tE}{\mathrm{E}}
\newcommand{\tF}{\mathrm{F_4}}
\newcommand{\tG}{\mathrm{G_2}}

\newcommand{\Hy}{\mathbf H}

\newcommand{\G}{\mathrm G}
\newcommand{\bG}{\overline{\G}}
\newcommand{\CG}{\mathrm C}

\newcommand{\GG}{\mathbf G}

\newcommand{\CC}{\mathcal C}

\newcommand{\bs}{\backslash}
\newcommand{\del}{\partial}

\newcommand{\W}{\widetilde{W}}

\newcommand{\card}{\mathrm{card}}
\newcommand{\Tr}{\mathrm{Tr}}
\newcommand{\No}{\mathcal{N}}

\renewcommand{\a}{\mathfrak{a}}

\newcommand{\norm}[1]{|\!|#1|\!|}







\section{Introduction}
\label{sec:intro}

Let $F$ be a number field of degree $d$, with ring of integers $\O_F$ and discriminant
$D_F$. We denote by $K_n(\O_F)$ 
the $n$-th $K$-theory group of $\O_F$, which was defined by Quillen and showed by him 
to be finitely generated. The rank of $K_n(\O_F)$ has been computed by Borel in
\cite{Bor77}. 
In this article we consider the problem of finding  
an upper bound  -- in terms of $n$, $d$ and $D_F$ -- for the order of the torsion part
$K_n(\O_F)_\tors$. Such general bounds have been obtained by Soul\'e in \cite{Soule03}. 
Our Theorem~\ref{thm:Kn} below sharpens Soul\'e's results.

As in Soul\'e's paper, our inequalities hold ``up to small torsion''. To state this precisely, for a
finite abelian group $A$ let us write $\card_\ell(A)$  for the order of $A/B$, where $B \subset
A$ is the subgroup generated by elements of order $\le \ell$.

\begin{theorem}
  \label{thm:Kn}
  Let $F \neq \Q$ be a number field of degree $d$ and with discriminant
  $D_F$. Then for any $n \ge 2$ we have
  \begin{align*}
          \log \, \card_\ell\, K_n(\O_F)_\tors &\le (2n+1)^{71 n^4 d^3} \cdot
          d^{293 n^5 d^5} \cdot |D_F|^{528 n^5 d^4} \,,
  \end{align*}
  where $\ell = \max(d+1, 2n+2)$.
\end{theorem}

To improve the readability, we have not tried to state here the best possible bounds that
one could get with the method we use. We refer to the PhD thesis of the third author
\cite[Theorem~4.3]{Houriet}
-- in which the result was originally obtained -- for slightly better estimates. However, it does
not change the fact that the upper bounds are huge, and -- although explicit --  
certainly unusable for practical computation. We shall insist here on
the qualitative aspect of our result, which could be stated as follows.

\begin{cor}
        \label{cor:Kn}
        There exist $\alpha$ and  $\beta$, both polynomials in $n$ and $d$,
        such that for any number field $F$ of degree $d \ge 2$ we have  
        \begin{align*}
                \log \, \card_\ell \, K_n(\O_F)_\tors &\le  (n d)^\alpha|D_F|^\beta, 
        \end{align*}
        where $\ell = \max(d+1, 2n+2)$ and $n \ge 2$.
\end{cor}

Compared to \cite{Soule03}, our result improves the bound by an exponential factor:  
the previous bound for $\log \card_\ell K_n(\O_F)$ was at least $\exp(\alpha
|D_F|^{1/ 2})$, for some polynomial $\alpha = \alpha(n,d)$ (see Proposition 4 in loc.\
cit.\ for the precise statement). 

The strategy is the following.
The group $K_n(\O_F)$ can be related -- via the Hurewicz map -- to the integral homology
$H_n(\GL(\O_F))$, and an upper bound (up to small torsion) for the order of $K_n(\O_F)$ can
then be obtained through the study of the integral homology of $\GL_N(\O_F)$, with $N=2n+1$ (cf.
Section~\ref{bound-homology-K-th}).
The proof of Theorem~\ref{thm:Kn} follows the method of Soul\'e, which uses Ash's
well-rounded retract (cf. Section~\ref{sec:hermitian-well-rounded}) to study these homology
groups. This reduces the problem to finding good estimates concerning the geometry of
hermitian lattices.
Our approach to these estimates differs from that of Soul\'e 
(cf. Section~\ref{sec:hermit-lattice-bounded-bases}), leading to the improved bounds
in Theorem~\ref{thm:Kn}.

For $F = \Q$ our method does not bring any improvement, so that
\cite[Prop.~4~iv)]{Soule03} is still the best available general bound for $K_n(\Z)$.
We refer to \cite[Theorem~1.3]{EmeK2} for a different approach to the
same problem for $K_2$, which gives better result than Corollary~\ref{cor:Kn} 
in the case of totally imaginary fields. Note that all
these results remain very far from the general bound conjectured by
Soul\'e in \cite[Sect.~5.1]{Soule03}, which should take the following form for
some constant $C(n, d)$: 
\begin{align}
  \log \, \card\, K_n(\O_F)_\tors &\le C(n, d) \, \log |D_F|. 
  \label{eq:conject-Soule}
\end{align}

\subsection*{Acknowledgement} We thank Christophe Soul\'e for suggesting
that an improvement of his bound in \cite{Soule03} could be possible with the methods
of the paper of the first named author~\cite{Bayer06}.

\section{Hermitian metrics and the well-rounded retract}
\label{sec:hermitian-well-rounded}

\subsection{Notation}


We keep the notation of the introduction. Let us denote 
by $(r_1, r_2)$ the signature of the number field $F$.
Let us  write $F_\R = \R \otimes_\Q F$. If $\Sigma$ denotes the set of field embeddings 
$\sigma: F \to \C$, then $F_\R$ can be identified with the subspace $(\C^\Sigma)^+ \subset
\C^\Sigma$ invariant under the involution $(x_\sigma) \mapsto
(\overline{x}_{\overline{\sigma}})$, where $\overline{a}$ denotes the complex
conjugation. This also provides an isomorphism $F_\R \cong \R^{r_1} \times \C^{r_2}$.

For $x \in F_\R$, written as  $x = (x_\sigma)_{\sigma \in \Sigma}$, we denote by 
$\overline{x} = (\overline{x_\sigma})$
the result of the complex conjugation applied component-wise. 
We denote by $\Tr$ the trace  map from $F_\R$ to $\R$,
defined by $\Tr(x) = \sum_{\sigma \in \Sigma} x_\sigma$.
We will also use the absolute value of the norm map:  $\No(x) =
\prod_{\sigma \in \Sigma} |x_\sigma|$.

We fix a free $\O_F$-lattice $L$ of finite rank $N \ge 1$.
Let $V = F \otimes_{\O_F} L$ and $V_\R = \R \otimes_\Q V$, so that $V_\R$ 
can be seen as a (left) $F_\R$-module.  
Let $\Gamma$ be the group $\GL(L)$ of automorphism of $L$. By fixing a basis of $L$, we
have the identification $\Gamma = \GL_N(\O_F)$. Then $\Gamma$ is a discrete subgroup of
the reductive Lie group $\GL(V_\R) = \GL_N(F_\R)$. We shall denote the latter by $G$, and
we will let it act on $V_\R$ on the left (and similarly for $\Gamma$ on $L$).




\subsection{Hermitian metrics}

Let $h : V_\R \times V_\R \to F_\R$ be a hermitian form on $V_\R$, that is, $h =
(h_\sigma)_{\sigma \in \Sigma}$ is $F_\R$-linear
in the first variable and $h(y,x) = \overline{h(x,y)}$. 
The pair $(L, h)$ is called a \emph{hermitian lattice}.
When $x = y$, we also write $h(x)
= h(x,x)$. Note that $h(x)$ has only real components, and  we say that $h$ is positive 
definite if $h_\sigma(x,x) >0 $ for any nonzero $x \in V_\R$ and all $\sigma \in \Sigma$.

Let $X$ be the (topological) space of positive definite hermitian forms on $V_\R$. The group
$G = \GL(V_\R)$ acts transitively on $X$ in the following way: the
element $\gamma \in G$ maps the form
$h \in X$ to
\begin{align}
        (\gamma \cdot h)(x,y) = h(\gamma^{-1}x, \gamma^{-1}y).
        \label{eq:action-G-X}
\end{align}
The space $X$ can be identified with the set of
positive definite symmetric $N\times N$ matrices with coefficients in $F_\R$. Using this
identification, it is not difficult to see that $X$ is contractible.

To each $h \in X$ we associate the real quadratic form $q_h$ on $V_\R$ (seen as a real vector
space) defined  for $x \in V_\R$ by:
\begin{align}
        q_h(x) &= \Tr(h(x)).
        \label{eq:quadr-form}
\end{align}
Such a quadratic form $q_h: V_\R \to \R$ for $h \in X$ will be called a \emph{hermitian
metric}. Given $h$, we will denote by $\norm{\cdot}_h$ the norm on $V_\R$
induced by $q_h$. 

\subsection{Ash's well-rounded retract}
\label{sec:Ash-retract}

For $h \in X$ we set $m(L,h) \in \R_{>0}$ to be the minimum of $q_h(x)$ over
the nonzero $x \in L \subset V$, and define 
\begin{align}
        M(L,h) &= \left\{ x \in L \;|\; q_h(x) = m(L, h) \right\}.
        \label{eq:Mh}
\end{align}

\begin{definition}
        We say that $h \in X$ (or $(L,h)$) is \emph{well rounded} if $m(L,h) = 1$
        and $M(L,h)$ generates $V$ (as a vector space over $F$). 
\end{definition}

Let $\W \subset X$ be the subspace of well-rounded hermitian forms. Note that the action
defined by \eqref{eq:action-G-X} restricts to an action of $\Gamma = \GL_N(\O_F)$ on $\W$.
In \cite[p. 466--467]{Ash}, Ash defined a CW-complex structure on $\W$ that has the following properties:
two points $h$ and $h'$ of $\W$ belong to the interior of the same cell $C(h) = C(h')$
if and only if $M(L,h) = M(L,h')$, and $C(h') \subset C(h)$ if and only if $M(L,h')
\supset M(L,h)$. Moreover, we have that $M(L,\gamma\cdot h) = \gamma M(L,h)$.
Then the action of $\Gamma$ on $\W$ is compatible with the cell structure: an element
$\gamma \in \Gamma$ maps the cell $C(h)$ to $C(\gamma \cdot h)$.
\begin{theorem}[Ash]
       $\W$ is a deformation retract of $X$ on which $\Gamma$ acts with finite
       stabilizer $\Gamma_\sigma$ for each cell $\sigma$ of $\W$. 
       The quotient $\Gamma\bs \W$ is compact, of dimension $\dim(X)-N$.
        \label{thm:Ash}
\end{theorem}
\begin{proof}
The proof of this statement follows from the argument of Ash given
in the proof of the main theorem of \cite{Ash}, page 462. 
More precisely, the argument to prove
that $\W$ is a deformation retract of $X$ is the same as Ash uses for 
$W = \W/\Gamma$, in \cite[\S 3 (i)]{Ash}.
The compactness of $\Gamma\bs\W$ is proved in \S 3 (ii) of loc.\ cit., 
and the dimension is computed on page 466. 
\end{proof}

Let $\CC_\bullet$ be the complex  of cellular chains on
$\Gamma\bs\W$. We can decompose it as $\CC_\bullet = \CC_\bullet^+
\cup \CC_\bullet^-$, where $\Gamma$ preserves (resp. does not preserve) the
orientation of each $\sigma \in \CC_\bullet^+$ (resp. $\sigma \in \CC_\bullet^-$). 
It then follows from the spectral sequence described in \cite[VII~(7.10)]{Brown82}
that up to prime divisors of the finite stabilizers
$\Gamma_\sigma$, the homology of $\CC_\bullet^+$ computes $H_\bullet(\Gamma)$.
In particular, one has the following (cf. \cite[Lemma~9]{Soule03}).
See Section~\ref{sec:intro} for the definition of $\card_\ell$. 

\begin{cor}
  \label{cor:H-Gamma-as-H-C}
  Let $\ell = 1 + \max(d,N)$.   Then  for any $n$ we have:
  \begin{align*}
      \card_\ell\, H_n(\Gamma)_\tors  = \card_\ell\,
      H_n(\CC_\bullet^+)_\tors,
  \end{align*}
  where $H_n(\cdot)_\tors$ denotes the torsion part of the integral
  homology.
\end{cor}

\section{Bounding torsion homology and $K$-theory}
\label{bound-homology-K-th}

\subsection{The Hurewicz map}
\label{sec:Hurewicz-map}

For any $n>1$ we
consider the $n$-th Quillen $K$-group $K_n(\O_F) = \pi_n(B \GL(\O_F)^{+})$
(``plus construction'').  
The Hurewicz  map relating homotopy groups to homology provides 
a map $K_n(\O_F) \to H_n(\GL(\O_F)^+) = H_n(\GL(\O_F))$.
We know (see for instance \cite[Theorem~1.5]{Arl91}) that its kernel does not contain
elements of order $p$ for $p > \frac{n+1}{2}$. Moreover, 
by a stability result of van der Kallen and Maazen (cf. \cite[Theorem~4.11]{vdK80})
the homology of
$\GL(\O_F) = \varinjlim \GL_N(\O_F)$ is equal to the homology of
$\GL_N(\O_F)$ for any $N \ge 2n + 1$.  
Let then $N = 2n+1$, and consider $\Gamma = \GL_N(\O_F)$. We deduce from
Corollary~\ref{cor:H-Gamma-as-H-C} that for $\ell = \max(d+1, 2n+2)$ we
have:

\begin{align}
  \card_\ell \, K_n(\O_F)_\tors &\le \card_\ell \,
  H_n(\CC^+_\bullet)_\tors. 
  \label{eq:bound-Kn-H-C}
\end{align}

\subsection{Gabber's lemma}
\label{sec:Gabber-lemma}

The abstract result that allows to obtain a bound for the right hand side of 
\eqref{eq:bound-Kn-H-C} is the following lemma.  It was discovered by Gabber,
and first appeared in Soul\'e \cite[Lemma~1]{Soule99}. 
See Sauer \cite[Lemma~3.2]{Sauer} for a more elementary proof. 

\begin{lemma}[Gabber]
 Let  $A = \Z^a$ with the standard basis $(e_i)_{i=1,\dots,a}$
 and $B = \Z^b$, so that $B \otimes \R$ is equipped with the standard
 Euclidean norm $\|\cdot\|$.  Let $\phi: A \to B$ be a $\Z$-linear map
 and let $\alpha \in \R$
 such that  $\| \phi(e_i) \| \le \alpha$ for each $i = 1, \dots, a$. 
 If we denote by $Q$ the cokernel of $\phi$, then 
 \begin{align*}
   |Q_\tors| &\le  \alpha^{\min(a,b)}\;.
 \end{align*}
       \label{lemma:Gabber}
\end{lemma}

\begin{cor}
  \label{cor:H-C}
  Suppose that the cellular complex $\Gamma\bs \W$ has at most
  $\alpha_k$ faces for any $k\ge 0$, and that any $k$-cell has at most
  $\beta$ codimension 1 faces. Then
  \begin{align*}
    H_k(\CC^+_\bullet)_\tors \le \beta^{\frac{1}{2} \min(\alpha_{k+1}, \alpha_k)}.
  \end{align*}
\end{cor}

\begin{proof}
        For a cell $c \in \CC_{k+1}^+$, its image by the boundary map $\partial$ is a sum of 
        at most $\beta$ $k$-cells, so that $\norm{\partial c} \le \sqrt{\beta}$. Thus, 
        by Lemma~\ref{lemma:Gabber}
        $\mathrm{coker}(\partial)_\tors$ is bounded by $\beta^{\frac{1}{2} 
        \min(\alpha_{k+1}, \alpha_k)}$ and a fortiori so is $H_k(\CC_\bullet^+)_\tors$. 
\end{proof}


\subsection{Counting the cells}
\label{sec:counting-cells}

Suppose that the finite subset $\Phi \subset L$ has the following property:
\begin{quote}
for any well-rounded pair $(L, h)$, there exists $\gamma \in \Gamma =
\GL_N(\O_F)$ such that $\gamma M(L,h) \subset \Phi$. 
\end{quote}
In other words, $\Phi$ contains a representative of every
element of $\Gamma\bs \W$. Since $C(h)$ has
codimension $j$, where $N+j$ is the cardinality of $M(L,h)$, it follows
immediately that the number of cells of codimension $j$ in $\Gamma\bs\W$
is bounded by the binomial coefficient $\binom{\card(\Phi)}{N+j}$.
For large $\card(\Phi)$ we lose little by bounding this binomial
coefficient by $\card(\Phi)^{N+j}$. Recall that $\Gamma\bs\W$ has dimension 
$\dim(X) - N $, so that for a $k$-cell of codimension $j$ we have
$N + j = \dim(X) -k$.  
For the dimension of $X$ we have (where $(r_1, r_2)$ is the signature of $F$):
\begin{align}
        \dim(X) &= r_1 \frac{N(N+1)}{2} + r_2 N^2 \\
        &\le d \, \frac{N(N+1)}{2}.
        \label{eq:dim-X}
\end{align}
Thus, for the number
of $k$-cells in $\Gamma\bs\W$ we can use the following upper bound:
\begin{align}
        \alpha_k &= \card(\Phi)^{d \cdot \frac{N(N+1)}{2} - k} 
  \label{eq:alpha_k-Phi}
\end{align}

By a similar counting argument, Soul\'e shows in \cite[proof of Prop.~3]{Soule03} that there are at most $\beta = \card(\Phi)^{N+1}$ faces of codimension $1$ in
any given cell (not necessarily top dimensional) on $\Gamma\bs\W$.

\subsection{Bounds for $K$-theory in terms of $\Phi$}
\label{sec:bounds-K-theory}

Let $\ell = \max(d+1,2n+2)$.
By Corollary~\ref{cor:H-C} and \eqref{eq:bound-Kn-H-C} we have that $\card_\ell \,
K_n(\O_F)_\tors$ is bounded by $\beta^{\frac{1}{2} \alpha_{n+1}}$, where $\alpha_{n+1}$
and $\beta$ can be chosen as in Section~\ref{sec:counting-cells} (with $N = 2n+1$).
This gives (using now logarithmic notation):
\begin{align}
        \log \card_\ell\, K_n(\O_F)_\tors &\le  (n+1)\, \log(\card(\Phi))\, 
        \card(\Phi)^{e(d,n)}, 
        \label{eq:bound-Kn-Phi}
\end{align}
where $\Phi \subset L$ has the property defined in Section~\ref{sec:counting-cells}, and   
\begin{align}
       e(d,n) &= d(2n^2 + 3n + 1) - n - 1. 
        \label{eq:e-dn}
\end{align}

This reduces the problem to finding such a set $\Phi \subset L$ of size as small as
possible. In \cite{Soule03} Soul\'e constructed a suitable set $\Phi$ using the geometry of numbers. 
In what follows, we will exhibit a smaller $\Phi$ by using better estimates on hermitian
lattices. 

\section{Hermitian lattices and bounded bases}
\label{sec:hermit-lattice-bounded-bases}

The goal of this section is to construct in any well-rounded
lattice $(L, h)$ a basis whose vectors have bounded length, with respect
to the norm induced by $h$. The method in an adaptation of the idea used by
Soul\'e in \cite{Soule03} (see Section~\ref{sec:bounded-bases} below),
in which we incorporate the results from \cite{Bayer06}, corresponding to 
the rank one case.

\subsection{Geometry of ideal lattices}
\label{sec:ideal-lattices}

Let $I \subset F_\R$ be a nonzero $\O_F$-submodule of the form $I = x \a$, where
$x \in F_\R$ and $\a$ is a fractional ideal of $F$. We define the
\emph{norm} of $I$ by the rule $\No(I) = \No(x) \No(\a)$.
Let $q_0$ be the standard (positive definite) hermitian metric on
$F_\R$, i.e., for $x \in F_\R$: $q(x) = \Tr(x \overline{x})$ .
The pair $(I, q_0)$ is an \emph{ideal lattice} (over $F$) in the sense
of \cite[Def.~2.2]{Bayer06}. Its determinant 
is given by (see \cite[Cor.~2.4]{Bayer06}):
\begin{align}
  \det(I, q_0) &= \No(I)^2 \, |D_F|. 
  \label{eq:det}
\end{align}

Let us denote by $\norm{\cdot}$ the norm on $F_\R$ induced by the hermitian metric
$q_0$.
Estimates for the geometry of ideal lattices have been studied in
\cite{Bayer06}. For our particular case $(I, q_0)$, Proposition 4.2 in
loc.\  cit.\ takes the following form. 
\begin{prop}
  \label{prop:eva}
  Let $F$ of degree $d$, with discriminant $D_F$, and consider the ideal
  lattice $(I, q_0)$.
  Then for any $x \in F_\R$ there exists 
  $y \in I$ such that $\norm{x-y} \le R$, where
  \begin{align*}
    R &= \frac{\sqrt{d}}{2} |D_F|^{1/d} \No(I)^{1/d}.
  \end{align*}
\end{prop}

\subsection{Three consequences}

From Proposition~\ref{prop:eva} we deduce the three following lemmas.

\begin{lemma}
  \label{lemma:C2}
  Given $x = (x_\sigma) \in F_\R$, there exists $a \in \O_F$ such that  
  \begin{align*}
    \sum_{\sigma \in \Sigma} |x_\sigma - \sigma(a)| &\le C_2,
  \end{align*}
  where
  \begin{align}
    C_2 &= \frac{d}{2}\, |D_F|^{1/d}.    
    \label{eq:C2}
  \end{align}
\end{lemma}

\begin{proof} 
  First note the general inequality $(\sum_{i=1}^d b_i)^2 \le d \cdot
  \sum_{i=1}^d b_i^2$, which follows from applying the summation $\sum_{i,j}$ on both
  sides of
  \begin{align*}
    2b_i b_j &\le b_i^2 + b_j^2.
  \end{align*}
  This implies that for $a = y \in I$ as in Proposition~\ref{prop:eva} with
  $I = \O_F$, we have
  \begin{align*}
    \sum_{\sigma \in \Sigma} |x_\sigma - \sigma(a)| &\le \sqrt{d \cdot
      \sum_{\sigma \in \Sigma} |x_\sigma - \sigma(a)|^2}\\
      &= \sqrt{d} \; \norm{x-a}\\
      &\le \frac{d}{2} |D_F|^{1/d}.  
  \end{align*}
\end{proof}

\begin{lemma}
  \label{lemma:C3}
  Given $x = (x_\sigma) \in F_\R$, there exists $a \in \O_F$ such that 
  \begin{align*}
    \sup_{\sigma \in \Sigma} |\sigma(a)\, x_\sigma | &\le C_3 \, \No(x)^{1/d}, 
  \end{align*}
  where
  \begin{align}
    C_3 &= \sqrt{d} \cdot |D_F|^{1/d}.
    \label{eq:C3}
  \end{align}
\end{lemma}

\begin{proof}
  We can suppose that $x \neq 0$. We consider the ideal lattice
  $(I, q_0)$ with $I = x \O_F$.
  For $R$ as in Proposition~\ref{prop:eva}, we have
  that $F_\R = I + B_{R}(0)$, where $B_{R}(0)$ is the closed ball of radius
  $R$ with respect to $\norm{\cdot}$. In particular, the smallest
  (nonzero) vector $xa \in I = x\O_F$ has length $\le 2R$. That is, there exists 
  $a \in \O_F$ such that 
  \begin{align*}
    \sup_{\sigma \in \Sigma} |\sigma(a)\, x_\sigma | &\le \norm{xa} \\
    &\le 2R;\\
  \end{align*}
  and the result follows.
\end{proof}

\begin{lemma}
  \label{lemma:representatives}
  Let $\a$ be an ideal of $\O_F$. Then there exists a set $\mathcal{R}
  \subset \O_F$ of representatives of $\O_F/\a$ such that for any $x \in
  \mathcal{R}$ we have
  \begin{align*}
    \sum_{\sigma \in \Sigma} |\sigma(x)| &\le  C_2\, \No(\a)^{1/d}.
  \end{align*}
\end{lemma}
\begin{proof}
 Let us consider the ideal lattice $(I, q) = (\a, q_0)$, and let
 $R$ be as in Proposition~\ref{prop:eva}. Then for any $x \in \O_F
 \subset F_\R$, there exists $y \in I$ such that $\norm{x-y} \le R$. 
 But $x - y \equiv x\; (I)$, so that the closed ball $B_R(0)$ contains a
 representative of each class of $\O_F/\a$. The inequality is then
 obtained as in the proof of Lemma~\ref{lemma:C2}. 
\end{proof}

\subsection{Existence of bounded bases}
\label{sec:bounded-bases}

\begin{lemma}[Soul\'e]
  \label{lem:basis-induction}
  Let $L = L_1 \oplus \cdots \oplus L_N$ be a decomposition of the
  hermitian lattice $(L, h)$ into rank one lattices,
  and suppose that each $L_i$ contains a vector $f_i$ with
  $|L_i/\O_F f_i| \le k$ and $\norm{f_i}_h \le k\lambda$, for some
  $k, \lambda > 1$.
  Then $L$ has a basis $e_1, \dots, e_N$ such that 
  \begin{align*}
    \norm{e_i}_h &\le \lambda \left( 1 + C_2 \right)^{t+1} k^{(d+1)(4N-1)},
  \end{align*}
  where $t = \lfloor \log_2(N) \rfloor + 1$.
\end{lemma}
\begin{proof}
  The statement and its proof is essentially contained in the proof of  
  \cite[Prop.~1]{Soule03}. The main difference is that our constant
  $C_2$ is now smaller than $C_2$ in loc.\ cit.  We can follow verbatim
  the same proof with the new $C_2$ except for the  use of
  Lemma 6 (needed in Lemma 7) of loc.\ cit., which must be replaced by
  Lemma~\ref{lemma:representatives}. Accordingly, the factor $(1 + C_2 \frac{r+3}{4})$ 
  (where $r=d$) is replaced by  $1 + C_2$.
\end{proof}

To obtain a bounded basis for $(L, h)$ we need to find elements $f_i$
that satisfy the condition of Lemma~\ref{lem:basis-induction}. This is
done in the following proposition.

\begin{prop}
  \label{prop:exist-f_i}
  Let $(L, h)$ be a free hermitian lattice over $\O_F$  of rank $N$,
  with $F \neq \Q$.
  We suppose that there exist $e_1, \dots, e_N \in L$ that span $V = F
  \otimes_{\O_F} L$ and such that $\norm{e_i}_h \le 1$ for $i =
  1,\dots,N$.
  Then there exists a decomposition $L = L_1 \oplus \dots \oplus L_N$
  and elements $f_i \in L_i$ such that: 
  \begin{align*}
    |L_i/f_i \O_F| &\le C_1 C_3^d\; ; \\
    \norm{f_i}_h &\le i C_1 C_2 C_3^d,
  \end{align*}
  where $C_1 = |D_F|^{1/2}$, and $C_2$ (resp. $C_3$) is defined in
  \eqref{eq:C2} (resp. \eqref{eq:C3}).
\end{prop}
\begin{proof}
  The proof proceeds by induction, and follows the line of argument of
  \cite[proof of Lemma~5]{Soule03}. Let $N = 1$. By Lemma 1 in loc.\
  cit., there exists $x \in L$ such that $|L/\O_F x| \le C_1 = |D_F|^{1/2}$.  
  Let us write $x = \alpha\cdot e_1$ for $\alpha \in F^\times$, where by
  assumption $\norm{e_1}_h \le 1$. By Lemma~\ref{lemma:C3} applied to
  $\alpha \in F_\R$, there exists $a \in \O_F$ such
  $\sup_\sigma|\sigma(a \alpha)| \le C_3 |N(\alpha)|^{1/d}$.  In
  particular, 
  \begin{align*}
    |N(a \alpha)| &\le \left(  \sup_{\sigma \in \Sigma} |\sigma(a
    \alpha)| \right)^d \\
    &\le C_3^d \, |N(\alpha)|,
  \end{align*}
  so that $|N(a)| \le C_3^d$. We set $f_1 = a \cdot x$. Then
  \begin{align*}
   |L/\O_F f_1| &= |N(a)| \cdot |L/\O_F x|\\
   &\le C_3^d C_1.
  \end{align*}
  For the norm we have:
  \begin{align*}
    \norm{f_1}_h^2 &= \Tr(h(f_1, f_1))\\
    &= \sum_{\sigma \in \Sigma} |\sigma(\alpha)|^2 \,
    h_\sigma(e_1, e_1)\\
    &\le \left(  \sup_{\sigma \in \Sigma} |\sigma(\alpha)| \right)^2  \; \norm{e_1}_h^2\\
    &\le C_3^2 \cdot |N(\alpha)|^{2/d}.
  \end{align*}
  Moreover, $|L/\O_F x| = |N(\alpha)| \cdot |L/\O_F e_1|$, so that
  $|N(\alpha)| \le C_1$. This shows that $\norm{f_1}_h \le C_3 C_1^{1/d}$
  and thus concludes the proof for $N=1$.

  The induction step is done exactly as in loc.\ cit., adapting the
  constants when necessary ($C_1$ to be replaced by $C_1 C_3^d$), to
  obtain the desired $f_i \in L_i$, i.e., with (using $F \neq \Q$ in the last
  inequality, so that $C_2 \ge 1$):
  \begin{align*}
    \norm{f_i}_h &\le (i-1) C_1 C_3^d C_2 + C_3 C_1^{1/d}\\
    &\le i C_1 C_2 C_3^d. 
  \end{align*}
\end{proof}

We finally obtain the result about the  existence of bounded bases. The
assumption $N \ge 5$ is only here in order to simplify the statement.

\begin{prop}
  \label{prop:bounded-basis}
  Let $(L, h)$ be a free hermitian lattice over $\O_F$ of rank $N \ge
  5$, with $F \neq \Q$, and such that the subset
  $\left\{ x \in L \;|\; \norm{x}_h \le 1  \right\}$ spans $V = F \otimes_{\O_F} L$.  
  Then there exists a basis $e_1, \dots, e_N$ of $L$ such that
  $\norm{e_i}_h \le B$ for every $i=1,\dots,N$, where
  \begin{align*}
          B &= \frac{4 N^2}{2^N}  \, d^{\, 5 N d^2} \, |D_F|^{6 N (d+1)}. 
  \end{align*}
\end{prop}
\begin{proof}
  By Proposition~\ref{prop:exist-f_i} we can apply Lemma~\ref{lem:basis-induction} with   
  \begin{align*}
    k &= C_1 C_3^d \; ;\\
    \lambda &= N \, C_2.
  \end{align*}
  This shows the existence of a basis $e_1, \dots, e_N$ with  
  \begin{align*}
    \norm{e_i}_h \le  N \, C_2 (1+C_2)^{\lfloor \log_2(N) \rfloor + 2} (C_1 C_3^d)^{(d+1)(4N-1)}.
  \end{align*}
  Since $C_2 \ge 1$, we have $(1+C_2)^n \le 2^n C_2^n$. Moreover, for
  $N \ge 5$ we have $\lfloor \log_2(N)\rfloor  + 3 \le N$. We deduce:
  \begin{align*}
    \norm{e_i}_h & \le  4 N^2 \, C_2^N (C_1 C_3^d)^{(d+1)(4N-1)}\\
    &= \alpha d^\beta |D_F|^\gamma, 
  \end{align*}
  with (using $N \ge 5$ and $d\ge 2$):
  \begin{align*}
    \alpha  &= 4 \; \frac{N^2}{2^N} \; ;\\
    \beta &= N + \frac{d}{2} (d+1) (4N -1)\\
    &\le 5 N d^2\; ; \\
    \gamma &= \frac{N}{d} + \frac{3}{2}(d+1)(4N-1)\\
    &\le 6 (d+1) N.
  \end{align*}
  This finishes the proof.
\end{proof}

\section{Improved estimates for $K$-groups}
\label{sec:proof-bounds}

\subsection{A Bounded set $\Phi$}
\label{bound-Phi}
The construction of a bounded set $\Phi \subset L$ will follow from this
proposition. 

\begin{prop}
  \label{prop:borne-coeff-T}
  Let $(L, h)$ be a free well-rounded $\O_F$-lattice of rank $N \ge 5$,
  with $F \neq \Q$. Let $e_1, \dots, e_N$ and $B$ be defined as in
  Proposition~\ref{prop:bounded-basis}, and for $x \in M(L, h)$
  write $x = \sum_i x_i e_i$, with $x_i \in \O_F$. Then for every $i =
  1, \dots, N$ we have:
  \begin{align*}
    \sum_{\sigma \in \Sigma} |\sigma(x_i)|^2 &\le  T, 
  \end{align*}
  where
  \begin{align*}
   T &=  N^{N d} d^{\frac{3}{2} N d + 1} B^{2(Nd -1)} |D_F|^{2N}.
  \end{align*}

\end{prop}
\begin{proof}
  Let $x \in M(L, h)$, i.e., $h(x) = 1$.
 For each $\sigma \in \Sigma$ let us consider the matrix $H_\sigma =
 (h_\sigma(e_i, e_j))$. Then the first argument in \cite[proof of Prop.~2]{Soule03},
 based on the Hadamard inequality for positive definite
 matrix, shows that 
 \begin{align*}
   |\sigma(x_i)|^2 &\le   \det(H_\sigma)^{-1} h_\sigma(x) \prod_{j\neq
   i} h_\sigma(e_j).
 \end{align*}
 Since $h_\sigma(e_j) \le \norm{e_j}_h^2 \le B^2$, and similarly
 $h_\sigma(x) \le 1$, we obtain: 
 \begin{align}
   \sum_{\sigma \in \Sigma} |\sigma(x_i)|^2 &\le B^{2(N-1)} \sum_{\sigma
   \in \Sigma} \det(H_\sigma)^{-1}.
   \label{eq:sum-det-1}
 \end{align}
 For $\sum_\sigma \det(H_\sigma)^{-1}$ we can write, using the Hadamard
 inequality:
 \begin{align}
   \nonumber
   \sum_{\sigma \in \Sigma} \det(H_\sigma)^{-1} 
   &= \sum_{\sigma \in \Sigma} \left( \prod_{\sigma' \neq \sigma} \det(H_{\sigma'})
   \prod_{\sigma' \in \Sigma} \det(H_{\sigma'})^{-1} 
   \right)
   \\
   \nonumber
   &\le \left( \sum_{\sigma \in \Sigma} \; \prod_{\sigma' \neq \sigma} \;
   \prod_{j=1}^N h_{\sigma'}(e_j) \right) \cdot
   \left( \prod_{\sigma \in \Sigma} \det(H_\sigma)^{-1} \right)
   \\
   &  \le d \cdot B^{2 N (d-1)} \prod_{\sigma \in \Sigma}
   \det(H_\sigma)^{-1}.
   \label{eq:52}
 \end{align}
 According to Icaza \cite[Theorem~1]{Ica96}, there exists $z \in L$ such that 
 \begin{align}
   \prod_{\sigma \in \Sigma} \det(H_\sigma)^{-1} &\le  \gamma^N
   \No(h(z))^{-N},
   \label{eq:Icaza}
 \end{align}
 where (cf. \cite[Equ.~(21)]{Soule03}):
 \begin{align*}
   \gamma &\le N^d |D_F|. 
 \end{align*}
 By applying Lemma~\ref{lemma:C3} to $h(z) \in F_\R$, we find $a \in
 \O_F$ such that $h_\sigma(az) =  \sigma(a) h_\sigma(z) \le C_3
 \No(h(z))^{1/d}$ for every $\sigma \in \Sigma$.
 Since $(L, h)$ is well rounded, this implies:
 \begin{align}
   d\,  C_3\, \No(h(z))^{1/d} \ge h(a z) \ge 1,
   \label{eq:h-a-z}
 \end{align}
 so that $\No(h(z))^{-1} \le d^d C_3^d = d^{\frac{3}{2}d} |D_F|$.
 Using this with \eqref{eq:sum-det-1}, \eqref{eq:52} and \eqref{eq:Icaza}, this
 concludes the proof. 
\end{proof}

\begin{cor}
  \label{cor:Phi}
  Let $L$ be a free $\O_F$-lattice of rank $N \ge 5$, with $F \neq \Q$.
  Then there exists a subset $\Phi \subset L$ with the property given
  in Section~\ref{sec:counting-cells} and such that
  \begin{align*}
          \card(\Phi) \le N^{3 N^2 d^2} \cdot d^{5 N^3 d^4} 
                          \cdot |D_F|^{9 N^3 d^3}\,.
  \end{align*}
\end{cor}
\begin{proof}
  Let $f_1,\dots, f_N$ be any basis of $L$, and set, for $T$ as in
  Proposition~\ref{prop:borne-coeff-T}: 
  \begin{align*}
    \Phi &=  \left\{ \sum_{i=1}^N x_i f_i\; \Big|\; x_i \in \O_F \mbox{ with }
    \sum_{\sigma \in \Sigma} |\sigma(x_i)|^2 \le T \right\}. 
  \end{align*}
  According to \cite[Lemma~8]{Soule03}, the number of elements $x_i \in \O_F$ with 
  $\sum_{\sigma \in \Sigma} |\sigma(x_i)|^2 \le T$ is at most $T^{d/2} 2^{d+3}$, so that
  $\card(\Phi)$ is bounded above by $T^{Nd/2} 2^{N(d+3)}$. Expanding
  the constants $T$ and $B$ as in the statements of
  Propositions~\ref{prop:borne-coeff-T} and \ref{prop:bounded-basis},
  we obtain  the stated upper bound for
  $\card(\Phi)$.

  Let $h$ be a well-rounded hermitian metric on $L$. We can apply 
  Proposition~\ref{prop:borne-coeff-T} to write every $x \in M(L,h)$ as $x = \sum x_i e_i$ for a bounded
  basis $e_1, \dots, e_N$ of $L$. The proposition implies that  the
  transformation $\gamma \in \Gamma = \GL_N(\O_F)$
  that sends the basis $(e_i)$ to $(f_i)$ is such that $\gamma \cdot x \in \Phi$.
  This means that $\Phi$ has the property defined in Section~\ref{sec:counting-cells}.
\end{proof}

\subsection{Upper bounds for $K_n(\O_F)$}
 
We finally come  to the bounds for the $K$-groups of $\O_F$, as stated in Theorem~\ref{thm:Kn}. Let
$\ell = \max(d+1, 2n+2)$. From Equation~\eqref{eq:bound-Kn-Phi} we obtain
\begin{align*}
        \log\, \card_\ell\, K_n(\O_F)_\tors &\le \card(\Phi)^{e(d,n) + n + 1},
\end{align*}
and note that for $n\ge 2$ we have $e(d,n) + n + 1 \le \frac{15}{4} n^2 d$. Theorem~\ref{thm:Kn}
now follows directly from Corollary~\ref{cor:Phi}, applied with $N = 2n+1 \le \frac{5}{2} n$.

\label{sec:bounds-Kn}

\bibliography{Kn.bbl}


\Addresses
\end{document}